\documentclass{article}

\usepackage{amssymb,amsmath,amsthm}
\usepackage{graphicx}
\usepackage{proof}
\usepackage{cite}
\usepackage[UKenglish]{babel}
\usepackage{color}
\usepackage[all]{xy} 
\usepackage{hyperref}
\usepackage{times}

\hypersetup{
   colorlinks,%
   citecolor=blue,%
   filecolor=black,%
   linkcolor=red,%
   urlcolor=black
 }

\newcommand{\lr}[1]{\langle #1 \rangle}
\newcommand{\lra}{\leftrightarrow}

\newcommand{\bis}{\mathrel{\mathchoice%
{\raisebox{.3ex}{$\,
  \underline{\makebox[.7em]{$\leftrightarrow$}}\,$}}%
{\raisebox{.3ex}{$\,
  \underline{\makebox[.7em]{$\leftrightarrow$}}\,$}}%
{\raisebox{.2ex}{$\,
  \underline{\makebox[.5em]{\scriptsize$\leftrightarrow$}}\,$}}%
{\raisebox{.2ex}{$\,
  \underline{\makebox[.5em]{\scriptsize$\leftrightarrow$}}\,$}}}}

\newcommand{\C}{\textbf{C}}

\newcommand{\BP}{\textbf{P}}
\newcommand{\Ag}{\textbf{I}}

\newcommand{\M}{\ensuremath{\mathcal{M}}}
\newcommand{\N}{\ensuremath{\mathcal{N}}}

\renewcommand{\phi}{\varphi}
\renewcommand{\hat}{\widehat}

\newcommand{\weg}[1]{}

\newtheorem{theorem}{Theorem}
\newtheorem{lemma}[theorem]{Lemma}
\newtheorem{definition}[theorem]{Definition}

\newtheorem{claim}[theorem]{Claim}
\newtheorem{proposition}[theorem]{Proposition}

\newtheorem{corollary}[theorem]{Corollary}
\newtheorem{fact}[theorem]{Fact}

\begin{document}

\title{A road to ultrafilter extensions\thanks{This research is funded by the project 17CZX053 of National Social Science Fundation of China. The author would like to thank two anonymous referees of CLAR 2018 for their insightful comments. An earlier version of the manuscript was presented on the conference of CLAR 2018 in Hangzhou in June 2018 and included in the informal proceedings of the conference.}}

\date{}
\author{Jie Fan\\
School of Humanities, University of Chinese Academy of Sciences, China,\\
jiefan@ucas.ac.cn}

\maketitle

\begin{abstract}
We propose a uniform method of constructing ultrafilter extensions from canonical models, which is based on the similarity between ultrafilters and maximal consistent sets. This method can help us understand why the known ultrafilter extensions of models for normal modal logics and for classical modal logics are so defined. We then apply this method to obtain ultrafilter extensions of models for Kripke contingency logics and for neighborhood contingency logics.
\end{abstract}

\noindent Keywords: ultrafilter extensions, canonical models, ultrafilters, maximal consistent sets, model theory

\section{Introduction}

The notion of ultrafilter extensions goes back to Stone's representation theorem\weg{Stone representation for Boolean algebras}~\cite{Stone:1936} and the J\'{o}nsson-Tarski theorem~\cite{JonssonTarkski:1951}, and it is introduced in~\cite{vanbenthem:1979} (see~\cite[p.~372]{ChagrovZakharyaschev:1997}). As a classical result in model theory, ultrafilter extensions have played an important role in various non-classical logics, such as modal logic, temporal logic, dynamic logic, and so on. For instance, by using a notion of ultrafilter extension, \cite{vanbenthem:1979} semantically characterizes a kind of complete modal logics, called `canonical modal logics' introduced in~\cite{Fine:1975}. Besides, it is a crucial concept in various important results, such as the above-mentioned J\'{o}nsson-Tarski theorem, a second bisimilarity-somewhere-else result, Goldblatt-Thomason Theorem, van Benthem Characterization Theorem, and many others~(e.g.~\cite{GoldblattThomason:1975,vanbenthem:1984,vanbenthem:1988,Venema:1999,blackburnetal:2001,Kupkeetal:2005,Fanetal:2014,Bakhtiarietal:2017}).




Usually, a suitable notion of ultrafilter extensions should have the following nice properties (e.g.~\cite{blackburnetal:2001}):
\begin{itemize}
\item The ultrafilter extension and its original model are logically equivalent;
\item Ultrafilter extensions are saturated in some proper sense, i.e. in the sense that the class of saturated models has the Hennessy-Milner property: logical equivalence is equal to bisimilarity.
\end{itemize}

The two properties follows that
\begin{itemize}
\item Logical equivalence can be characterized as bisimilarity-somewhere-else, that is, between ultrafilter extension.
\end{itemize}


Compared to other model operations like disjoint unions, general submodels, and bounded morphisms, constructing ultrafilter extensions is seen as a far less natural job. Despite the fact that there are various notions of ultrafilter extensions of models, for instance, \cite{Goranko:2007,Saveliev:2011} for first-order logic, \cite{GoldblattThomason:1975} for normal modal logics, \cite{DBLP:journals/corr/abs-0901-4430} for classical modal logics, \cite{Jacobs:2001,Kupkeetal:2005} for coalgebras\weg{coalgebraic modal logics}, to our knowledge, however, there has been no uniform method to constructing ultrafilter extensions in the literature, and it is always seen to be hard to find a suitable notion of ultrafilter extensions (e.g.~\cite{blackburnetal:2001}). This paper proposes a uniform method of constructing ultrafilter extensions of models.\footnote{Note that ultrafilter extensions of {\em frames} follow directly by leaving out the valuations.}

The method is via a two-step transformation from the notion of canonical models. That is, given a canonical model, we can transform it into a desired ultrafilter extension in two steps. This is due to the similarity between ultrafilters and maximal consistent sets. For instance, an ultrafilter contain the whole domain, is closed under intersection and supersets, does not contain the empty set, and contains exactly one of any set and its complement; and a maximal consistent set contains the tautologies, is closed under conjunction and logical implication, does not contain the contradictions, and contains exactly one of any formula and its negation. It is these similar properties that makes the transformation from a canonical model to the corresponding ultrafilter extension workable. As we shall see, via this method, we can construct the desired ultrafilter extensions quite easily, in an automatic way.\footnote{It should be noted that~\cite{Goranko:2007} introduces a construction of ultrafilter extensions for arbitrary structures from universal-algebraic perspective, which is quite different from ours. Their construction is more like a definition based on ultrafilter extensions of the domain, relations and functions of structures, rather than a (transformation) method. Besides, as the author himself remarked, the `approaches' to ultrafilter extensions of functions either lead to complications in the algebraic theory of ultrafilter extensions, or is not completely satisfactory and have at least two natural rivals for the title `ultrafilter extensions' of functions.}

The structure of the paper is outlined as follows. After the basics of ultrafilters and maximal consistent sets followed by an important theorem (Sec.~\ref{sec.preliminaries}), we illustrate the method of constructing the known notions of ultrafilter extensions of models for normal modal logics (Subsec.~\ref{sec.normalmodal}) and for classical modal logics (Subsec.~\ref{sec.classicalmodal}). Then we apply the method to contingency logic, and obtain the ultrafilter extensions of Kripke models (Subsec.~\ref{sec.kripkecontingency}) and of neighborhood models (Subsec.~\ref{sec.neighborhoodcontingency}). We conclude with some discussions in Sec.~\ref{sec.conclusion}.

\section{Preliminaries}\label{sec.preliminaries}

Throughout the paper, we use $\overline{X}$ to denote the complement of the set $X$.

\begin{definition}[Ultrafilters]\label{def.ultrafilter} Let $S$ be a nonempty set. A set $U\subseteq \mathcal{P}(S)$ is an {\em ultrafilter} over $S$, if $U$
\begin{enumerate}
\item[(i)] Contains the whole set: $S\in U$.
\item[(ii)] Closed under intersection: if $X,Y\in U$, then $X\cap Y\in U$.
\item[(iii)] Closed under supersets: if $X\in U$ and $X\subseteq Z\subseteq S$, then $Z\in U$.
\item[(iv)] Does not contain empty set: $\emptyset\notin U$.
\item[(v)] for all $X\subseteq S$, $X\in U$ iff $\overline{X}\notin U$.
\end{enumerate}
\end{definition}

There is an important class of ultrafilters, called `principal ultrafilters'. Given a nonempty set $S$ and an element $w\in S$, the {\em principal ultrafilter} $\pi_w$ generated by $w$ is the filter generated by the singleton set $\{w\}$; in symbol, $\pi_w=\{X\subseteq S\mid w\in X\}$.

\begin{theorem}[Ultrafilter Theorem] Any proper filter over $S$ can be extended to an ultrafilter over $S$. As a corollary, any subset of $\mathcal{P}(S)$ with the finite intersection property can be extended to be an ultrafilter over $S$.
\end{theorem}

\begin{definition}[Maximal consistent sets] Let $\Sigma$ be a set of formulas. $\Sigma$ is said to be consistent, if $\Sigma\nvdash \bot$; it is said to be maximal, if for all formulas $\phi$, we have $\phi\in \Sigma$ or $\neg\phi\in \Sigma$; it is said to be {\em maximal consistent}, if it is maximal and also consistent.\footnote{Strictly speaking, notions of consistency and maximality, respectively, refer to a proof system and a language. But we leave out the references for simplicity.}
\end{definition}

Here is a list of some properties of maximal consistent sets (not exclusively). We choose these properties but not others, is because we would like to make clear the similarities between ultrafilters and  maximal consistent sets.\footnote{Note that there is a difference between the properties of these two notions: although it holds that $\phi\vee\psi\in\Sigma$ iff $\phi\in \Sigma$ or $\psi\in\Sigma$, it is {\em not} the case that $X\cup Y\in U$ iff $X\in U$ or $Y\in U$. The author would like to thank Christoph Benzm\"{u}ller for posing a related question and reminding the author of the difference on the conference of CLAR 2018.}

\begin{fact}[Properties of maximal consistent sets]\label{fact.mcs} Let $\Sigma$ be a maximal consistent set. Then
\begin{enumerate}
\item[(i)] $\top\in \Sigma$.
\item[(ii)] if $\phi,\psi\in\Sigma$, then $\phi\land\psi\in\Sigma$.
\item[(iii)] if $\phi\in \Sigma$ and $\vdash\phi\to\psi$, then $\psi\in \Sigma$.
\item[(iv)] $\bot\notin \Sigma$.
\item[(v)] for all formulas $\phi$, $\phi\in\Sigma$ iff $\neg\phi\notin \Sigma$.
\end{enumerate}
\end{fact}

\begin{lemma}[Lindenbaum's Lemma]
Every consistent set can be extended to a maximal consistent set.
\end{lemma}

Let $\circledcirc$ be an arbitrary modal operator, and $\mathcal{L}(\circledcirc)$ be the extension of the language of classical propositional logic $\mathcal{L}$ enriched with the primitive modality $\circledcirc$.

\begin{theorem}\label{prop.general}
Let $\Sigma$ be a maximal consistent set. If for all $\phi\in\mathcal{L(\circledcirc)}$, we have
$$(\ast)~~~~~~\text{for each }p\in\BP\cup\{\circledcirc \phi\}, p\in \Sigma\text{ iff }\Sigma\in V(p),$$
then for all $\phi\in\mathcal{L(\circledcirc)}$, we have
$$\phi\in\Sigma\text{ iff }V(\phi)\in \pi_\Sigma.$$
\end{theorem}

\begin{proof}
Suppose $(\ast)$ holds. We proceed with induction on $\phi\in\mathcal{L}(\circledcirc)$.
\begin{itemize}
\item $\phi=p\in\BP\cup\{\circledcirc \psi\}$. By $(\ast)$, $p\in\Sigma$ iff $\Sigma\in V(p)$. This is equivalent to $V(p)\in \pi_\Sigma$ according to the definition of $\pi_\Sigma$.
\item $\phi=\neg\psi$. We have the following equivalences:
\[
\begin{array}{ll}
&\neg\phi\in \Sigma\\
\stackrel{\text{Fact~}\ref{fact.mcs}(v)}\iff &\phi\notin \Sigma\\
\stackrel{\text{IH}}\iff&V(\phi)\notin \pi_\Sigma\\
\stackrel{\text{Def.~}\ref{def.ultrafilter}(v)}\iff&\overline{V(\phi)}\in \pi_\Sigma\\
\iff&V(\neg\phi)\in \pi_\Sigma\\
\end{array}
\]
\item $\phi=\psi\land\chi$. We have the following equivalences.
\[
\begin{array}{ll}
&\psi\land\chi\in \Sigma\\
\stackrel{\text{Fact~}\ref{fact.mcs}(ii)(iii)}\iff &\psi\in \Sigma\text{ and }\chi\in\Sigma\\
\stackrel{\text{IH}}\iff&V(\psi)\in \pi_\Sigma\text{ and }V(\chi)\in \Sigma\\
\stackrel{\text{Def.~}\ref{def.ultrafilter}(ii)(iii)}\iff&V(\psi)\cap V(\chi)\in \pi_\Sigma\\
\iff&V(\psi\land\chi)\in \pi_\Sigma\\
\end{array}
\]
\end{itemize}
\end{proof}

\weg{Since all principal ultrafilters are ultrafilters (cf. e.g.~\cite{blackburnetal:2001}), we have
\begin{proposition}
Let $\Sigma$ be a maximal consistent set. If
$$(\ast)~~~~~~\text{for each }p\in\BP, p\in \Sigma\text{ iff }\Sigma\in V(p),$$
then there exists an ultrafilter $U$ such that for all $\phi\in\mathcal{L}$, we have
$$\phi\in\Sigma\text{ iff }V(\phi)\in U.$$
\end{proposition}}

The relationship between maximal consistent sets and ultrafilters provides us with a road to the construction of ultrafilter extensions from that of canonical models, as will be illustrated in the setting of modal logic.

\section{Examples: ultrafilter extensions in modal logic}
\subsection{Normal modal logics}\label{sec.normalmodal}

Familiarity with the language $\mathcal{L}(\Box)$, Kripke semantics of modal logic, and also normal modal logics is assumed~(cf.~e.g.~\cite{blackburnetal:2001}). Given any nonempty set $S$, define $m_\Box(X)=\{s\in S\mid \text{for all }t,\text{ if }Rst, \text{ then }t\in X\}$. Then $m_\Box(V(\phi))=V(\Box\phi)$ for all $\phi$.

Recall that given a normal modal logic, the canonical model $\M=\lr{S,R,V}$ is defined as follows:\footnote{In this paper, we abuse the notation $\M=\lr{S,R,V}$ to denote an arbitrary model and also the canonical model. Unless otherwise specified, we always use it to mean an arbitrary model, thus there should be no confusion.}
\begin{itemize}
\item $S=\{\Sigma\mid \Sigma\text{ is a maximal consistent set}\}$.
\item $R \Sigma\Gamma$ iff for all $\phi$, if $\Box\phi\in\Sigma$, then $\phi\in\Gamma$.
\item For each $p\in\BP$ and $\Sigma\in S$, $\Sigma\in V(p)$ iff $p\in \Sigma$.
\end{itemize}

The following is a standard result in normal modal logics (cf.~e.g.~\cite{blackburnetal:2001}).
\begin{proposition}Let $\Sigma\in S$. Then for all $\phi\in\mathcal{L}(\Box)$, for each $p\in\BP\cup\{\Box\phi\}$, we have
$$p\in\Sigma\text{ iff }\Sigma\in V(p).$$
\end{proposition}

Applying Prop.~\ref{prop.general}, we immediately have the following result.
\begin{proposition}\label{prop.box}
Let $\Sigma\in S$. Then for all $\phi\in\mathcal{L(\Box)}$, we have
$$\phi\in\Sigma\text{ iff }V(\phi)\in \pi_\Sigma.$$
\end{proposition}

Recall that maximal consistent sets are thought of as a state in the canonical model construction, and from every state one may generate a principal ultrafilter, we thus can obtain the definition of ultrafilter extension of normal modal logics from that of the above canonical model. We will explicate this by two steps.

At first step, from the canonical model above, we replace maximal consistent sets with the corresponding principal ultrafilters, by using Prop.~\ref{prop.box} and the fact that $V(\Box\phi)=m_\Box(V(\phi))$. We obtain a model $ue(\M)=\lr{Uf(S),R^{ue},V^{ue}}$, where
\begin{itemize}
\item $Uf(S)=\{\pi_\Sigma\mid \pi_\Sigma\text{ is an ultrafilter over }S\}$.
\item $R^{ue} \pi_\Sigma\pi_\Gamma$ iff for all $V(\phi)$, if $m_\Box(V(\phi))\in\pi_\Sigma$, then $V(\phi)\in\pi_\Gamma$.
\item For each $p\in\BP$ and $\pi_\Sigma\in Uf(S)$, $\pi_\Sigma\in V^{ue}(p)$ iff $V(p)\in \pi_\Sigma$.
\end{itemize}

At second step, we polish the above model, by generalizing each principal ultrafilter and $V(\phi)$ to an arbitrary ultrafilter and an arbitrary set, respectively. By doing so, we obtain the notion of ultrafilter extension in normal modal logics (e.g.~\cite[Def.~2.57]{blackburnetal:2001}).
\begin{itemize}
\item $Uf(S)=\{s\mid s\text{ is an ultrafilter over }S\}$.
\item $R^{ue} st$ iff for all $X$, if $m_\Box(X)\in s$, then $X\in t$.
\item For each $p\in\BP$ and $s\in Uf(S)$, $s\in V^{ue}(p)$ iff $V(p)\in s$.
\end{itemize}

It is then shown that the notion of ultrafilter extension is the required one in normal modal logics. That is, the constructed ultrafilter extension and the original model are $\mathcal{L}(\Box)$-equivalent; moreover, $\mathcal{L}(\Box)$-equivalence can be characterized as $\Box$-bisimilarity-somewhere-else --- namely, between ultrafilter extensions. For the proof details, refer to e.g.~\cite[Prop.~2.59, Thm.~2.62]{blackburnetal:2001}.

\subsection{Classical modal logics}\label{sec.classicalmodal}

When modal logic is interpreted on neighborhood semantics, we obtain a class of classical modal logics. In this part, we apply our method to construct the ultrafilter extension of a neighborhood model for modal logic.

Given a neighborhood model $\M=\lr{S,N,V}$, the necessity operator is interpreted in the following.
\[
\begin{array}{lll}
\M,s\Vdash\Box\phi&\iff&V(\phi)\in N(s)\\
\end{array}
\]
Where $V(\phi)=\{w\in S\mid \M,w\Vdash\phi\}$.

Define $m_N(X)=\{s\in S\mid X\in N(s)\}$. It is clear that $m_N(V(\phi))=V(\Box\phi)$.

Recall that in the completeness of classical modal logic ${\bf E}$ (cf. e.g.~\cite{Chellas:1980}), the canonical model $\M=\lr{S,N,V}$ is defined as follows:
\begin{itemize}
\item $S=\{\Sigma\mid \Sigma\text{ is a maximal consistent set}\}$.
\item $N(\Sigma)=\{|\phi|\mid \Box\phi\in \Sigma\}$, where $|\phi|=\{\Sigma\in S\mid \phi\in\Sigma\}$.
\item For each $p\in\BP$ and $\Sigma\in S$, $\Sigma\in V(p)$ iff $p\in\Sigma$.
\end{itemize}

The following is a standard result in classical modal logic.
\begin{proposition}\label{prop.comp-classical-modal} Let $\Sigma\in S$. Then for all $\phi\in\mathcal{L}(\Box)$, for each $p\in\BP\cup\{\Box\phi\}$, we have
$$p\in\Sigma\text{ iff }\Sigma\in V(p).$$
\end{proposition}

Applying Prop.~\ref{prop.general}, we immediately have the following result.
\begin{proposition}\label{prop.modal}
Let $\Sigma\in S$. Then for all $\phi\in\mathcal{L(\Box)}$, we have
$$\phi\in\Sigma\text{ iff }V(\phi)\in \pi_\Sigma.$$
\end{proposition}

Now applying the method in Sec.~\ref{sec.normalmodal}, we obtain the notion of ultrafilter extensions in classical modal logics~\cite[Def.~4.20]{DBLP:journals/corr/abs-0901-4430}.

\begin{definition}[Ultrafilter extensions]\label{def.ultraexten-classicalmodal} Let $\M=\lr{S,N,V}$ be a neighborhood model. The triple $ue(\M)=\lr{Uf(S),N^{ue},V^{ue}}$ is the {\em ultrafilter extension} of $\M$, if
\begin{itemize}
\item $Uf(S)=\{s\mid s\text{ is an ultrafilter over }S\}$
\item $N^{ue}(s)=\{\hat{X}\mid m_N(X)\in s\}$, where $\hat{X}=\{s\in Uf(S)\mid X\in s\}$.
\item For each $p\in \BP$ and $s\in Uf(S)$, $s\in V^{ue}(p)$ iff $V(p)\in s$.
\end{itemize}
\end{definition}

It is then shown that the above notion of ultrafilter extension is the required one in classical modal logics. That is, the constructed ultrafilter extension and the original model are $\mathcal{L}(\Box)$-equivalent; moreover, $\mathcal{L}(\Box)$-equivalence can be characterized as behavioural-equivalence-somewhere-else --- namely, between ultrafilter extensions. For the proof details, refer to e.g.~\cite[Lem.~4.24, Thm.~4.27]{DBLP:journals/corr/abs-0901-4430}.

\section{Applications: ultrafilter extension in contingency logic}
\subsection{Kripke contingency logics}\label{sec.kripkecontingency}

In this section, we construct the ultrafilter extension out of a Kripke model in contingency logic. Since in the literature of contingency logic, there are mainly two ways of defining the canonical model, this leads to two notions of ultrafilter extensions. 


The language of contingency logic is denoted $\mathcal{L}(\nabla)$, where $\nabla$ is read `it is contingent that', and the non-contingency operator $\Delta$ is abbreviated as $\neg\nabla$. Semantically, given a Kripke model $\M=\lr{S,R,V}$ and a state $s\in S$:
\[
\begin{array}{lll}
\M,s\vDash\nabla\phi&\iff&\text{there are }t,u\in S\text{ such that }Rst,Rsu\text{ and }\M,t\vDash\phi\text{ and }\M,u\nvDash\phi.\\
\end{array}
\]

And consequently, 
\[
\begin{array}{lll}
\M,s\vDash\Delta\phi&\iff&\text{for all }t,u\in S,\text{ if }Rst,Rsu,\text{ then }(\M,t\vDash\phi\iff\M,u\vDash\phi).\\
\end{array}
\]

A Kripke model $\M$ is said to be {\em $\mathcal{L}(\nabla)$-saturated}, if for any $s$ in $\M$, for any $\Gamma\subseteq\mathcal{L}(\nabla)$, if $\Gamma$ is finitely satisfiable in the set of successors of $s$, then $\Gamma$ is satisfiable in the set of successors of $s$.

Similar to the case of normal modal logics, we here need some requisite notation.

\begin{definition}\label{def.m-Delta} Given any nonempty set $S$,
\[
\begin{array}{lcl}
m_\nabla(X)&=&\{s\in S\mid Rst,Rsu\text{ for some }t\in X\text{ and some }u\notin X\}\\
m_\Delta(X)&=&\{s\in S\mid \text{for all }t,u,\text{ if }Rst\text{ and }Rsu, \text{ then }t\in X\text{ iff }u\in X\}.\\
\end{array}
\]
\end{definition}

It is straightforward to verify that $m_\nabla(V(\phi))=V(\nabla\phi)$ and $m_\Delta(V(\phi))=V(\Delta\phi)$. The definition of $m_\nabla$ is very natural, in that just as ``$Rst,Rsu\text{ for some }t\in X\text{ and some }\\
u\notin X$'' corresponds to the Kripke semantics of $\nabla$, $s\in m_\nabla(X)$ corresponds to the Kripke semantics of $\nabla$. The naturalness of $m_\Delta$ is similar. In fact, the definitions of $m_\nabla$ and $m_\Delta$ can be followed from the semantics of $\nabla\phi$ and $\Delta\phi$, by generalizing $V(\phi)$ to an arbitrary set $X$. Also, $m_\nabla(X)=m_\nabla(\overline{X})$, $m_\Delta(X)=m_\Delta(\overline{X})$ and $m_\Delta(X)=\overline{m_\nabla(X)}$. Moreover,


\begin{proposition}\label{prop.subseteq} For all $X,Y$, we have
\begin{enumerate}
\item[(a)] $m_\Delta(X)\cap m_\Delta(Y)\subseteq m_\Delta(X\cap Y)$.
\item[(b)] $m_\nabla(X\cap Y)\cap m_\nabla(X'\cap \overline{Y})\subseteq m_\nabla(Y)$.
\end{enumerate}
\end{proposition}

\begin{proof} Item (a) is direct from Def.~\ref{def.m-Delta}.

For item (b), let $s\in m_\nabla(X\cap Y)$ and $s\in m_\nabla(X'\cap \overline{Y})$. Then there are $t,u\in S$ such that $sRt$ and $sRu$ and $t\in X\cap Y$ and $u\notin X\cap Y$, and there are $t',u'\in S$ such that $sRt'$ and $sRu'$ and $t'\in X'\cap \overline{Y}$ and $u'\notin X'\cap \overline{Y}$. Now focus on $t$ and $t'$: $t\in Y$ and $t'\notin Y$. Therefore, $s\in m_\nabla(Y)$.
\end{proof}

\weg{The following is direct from Def.~\ref{def.m-Delta}.
\begin{proposition}\label{prop.intersection}
For all $X,Y$, $m_\Delta(X)\cap m_\Delta(Y)\subseteq m_\Delta(X\cap Y)$.
\end{proposition}}

\subsubsection{$\exists\forall$-version}

One version is inspired by the canonical model proposed in~\cite[Def.~5.6]{Fanetal:2014}, which is in turn inspired by an `almost-definability' schema $\nabla\chi\to(\Box\phi\lra\Delta\phi\land\Delta(\chi\to\phi))$~\cite[Def.~2.4]{Fanetal:2014}, see also~\cite[Prop.~3.5, Def.~4.5]{Fanetal:2015} for the multi-modal case.

\weg{Recall that the canonical model $\M=\lr{S,R,V}$ for minimal contingency logic in~\cite{Fanetal:2014} is defined as follows:}

Recall that the canonical model $\M=\lr{S,R,V}$ for minimal contingency logic in~\cite[Def.~5.6]{Fanetal:2014} (see also~\cite[Def.~4.5]{Fanetal:2015} for the multi-modal case) is defined as follows:
\begin{itemize}
\item $S=\{\Sigma\mid \Sigma\text{ is a maximal consistent set}\}$.
\item $R\Sigma\Gamma$ iff
      \begin{itemize}
      \item there exists $\chi$ such that $\nabla\chi\in \Sigma$, and
      \item for all $\phi$, if $\Delta\phi\land\Delta(\chi\to\phi)\in \Sigma$, then $\phi\in \Gamma$.
      \end{itemize}
\item For each $p\in\BP$ and $\Sigma\in S$, $\Sigma\in V(p)$ iff $p\in \Sigma$.
\end{itemize}

The following result is shown in~\cite[Lemma~5.7]{Fanetal:2014} (also see~\cite[Lemma~4.6]{Fanetal:2015} for the multi-modal case).
\begin{proposition}
Let $\Sigma\in S$. Then for all $\phi\in\mathcal{L}(\nabla)$, for each $p\in\BP\cup\{\nabla\phi\}$, we have
$$p\in\Sigma\iff \Sigma\in V(p).$$
\end{proposition}

Applying Prop.~\ref{prop.general}, we get the following result.
\begin{proposition}
Let $\Sigma\in S$. Then for all $\phi\in\mathcal{L}(\nabla)$, we have
$$\phi\in\Sigma\iff V(\phi)\in \pi_\Sigma.$$
\end{proposition}

Now with the method in Sec.~\ref{sec.normalmodal} in hand, we can obtain the ultrafilter extension in Kripke contingency logics. We call it `$\exists\forall$-version' since the accessibility relation $R^{ue}$ has the $\exists\forall$-form.

\begin{definition}[Ultrafilter Extension] Let $\M=\lr{S,R,V}$ be a model. The triple $ue(\M)=\lr{Uf(S),R^{ue},V^{ue}}$ is the {\em ultrafilter extension} of $\M$, if
\begin{itemize}
\item $Uf(S)=\{s\mid s\text{ is an ultrafilter over }S\}$.
\item $R^{ue}st$ iff there exists $X$ such that $m_\nabla(X)\in s$, and for all $Y$, if $m_\Delta(Y)\cap m_\Delta(\overline{X}\cup Y)\in s$, then $Y\in t$.
\item $V^{ue}(p)=\{s\in Uf(S)\mid V(p)\in s\}$.
\end{itemize}
\end{definition}

In what follows, we verify that the notion of ultrafilter extension constructed above is a desired one in Kripke contingency logics. That is, ultrafilter extension and the original model are $\mathcal{L}(\nabla)$-equivalent; moreover, $\mathcal{L}(\nabla)$-equivalence can be characterized as $\Delta$-bisimilarity-somewhere-else --- namely, between ultrafilter extensions.

\begin{proposition}\label{prop.ultrafilter1}
For all formulas $\phi\in\mathcal{L}(\nabla)$ and all ultrafilters $s$ over $S$, we have $V(\phi)\in s$ iff $ue(\M),s\vDash\phi$. As a corollary, for all $w\in S$, we have $(\M,w)\equiv_{\mathcal{L}(\nabla)}(ue(\M),\pi_w)$.
\end{proposition}

\begin{proof}
By induction on $\phi\in\mathcal{L}(\nabla)$. The only nontrivial case is $\nabla\phi$.

First suppose $ue(\M),s\vDash\nabla\phi$. Then there are $t,u\in Uf(S)$ such that $R^{ue}st$ and $R^{ue}su$ and $ue(\M),t\vDash\phi$ and $ue(\M),u\nvDash\phi$. By induction hypothesis, $V(\phi)\in t$ and $V(\phi)\notin u$. From $R^{ue}st$ it follows that there exists $X$ such that $m_\nabla(X)\in s$ and, for all $Y$, if $m_\Delta(Y)\cap m_\Delta(\overline{X}\cup Y)\in s$, then $Y\in t$. Since $V(\phi)\in t$ and $t$ is an ultrafilter, $\overline{V(\phi)}\notin t$. Then $m_\Delta(\overline{V(\phi)})\cap m_\Delta(\overline{X}\cup \overline{V(\phi)})\notin s$. Similarly, from $R^{ue}su$ we can obtain that there exists $X'$ such that $m_\Delta(V(\phi))\cap m_\Delta(\overline{X'}\cup V(\phi))\notin s$. We claim that $m_\nabla(V(\phi))\in s$: if not, then $m_\Delta(V(\phi))\in s$. Since $m_\Delta(V(\phi))=m_\Delta(\overline{V(\phi)})$, by the fact that ultrafilters are closed under intersection, we have $m_\Delta(\overline{X}\cup \overline{V(\phi)})\notin s$ and $m_\Delta(\overline{X'}\cup V(\phi))\notin s$. This means that $m_\nabla(X\cap V(\phi))\in s$ and $m_\nabla(X'\cap \overline{V(\phi)})\in s$. Again, by the fact that ultrafilters are closed under intersection, we have $m_\nabla(X\cap V(\phi))\cap m_\nabla(X'\cap \overline{V(\phi)})\in s$. Then using Prop.~\ref{prop.subseteq}(b) and the fact that ultrafilters are closed under supersets, we infer that $m_\nabla(V(\phi))\in s$: an contradiction. Therefore, $m_\nabla(V(\phi))\in s$, viz. $V(\nabla\phi)\in s$.

Now assume $V(\nabla\phi)\in s$, to show $ue(\M),s\vDash\nabla\phi$. By induction hypothesis, this means that we need to find two states $t,u$ in $Uf(S)$ such that $R^{ue}st$ and $R^{ue}su$ and $V(\phi)\in t$ and $V(\phi)\notin u$. By assumption and $m_\nabla(V(\phi))=V(\nabla\phi)$, we have $m_\nabla(V(\phi))\in s$.

Define $\Omega=\{Y\mid m_\Delta(Y)\cap m_\Delta(\overline{V(\phi)}\cup Y)\in s\}$. We first show that $\Omega$ is closed under intersection. Let $Y,Z\in\Omega$. Then $m_\Delta(Y)\cap m_\Delta(\overline{V(\phi)}\cup Y)\in s$ and $m_\Delta(Z)\cap m_\Delta(\overline{V(\phi)}\cup Z)\in s$. By the fact that ultrafilters are closed under supersets, we have $m_\Delta(Y), m_\Delta(\overline{V(\phi)}\cup Y)\in s$ and $m_\Delta(Z), m_\Delta(\overline{V(\phi)}\cup Z)\in s$. Then by the fact that ultrafilters are closed under intersection, $m_\Delta(Y)\cap m_\Delta(Z)\in s$ and $m_\Delta(\overline{V(\phi)}\cup Y)\cap m_\Delta(\overline{V(\phi)}\cup Z)\in s$. Again, by Prop.~\ref{prop.subseteq}(a) and the fact that ultrafilters are closed under supersets, $m_\Delta(Y\cap Z)\in s$ and $m_\Delta(\overline{V(\phi)}\cup (Y\cap Z))\in s$. Therefore $Y\cap Z\in \Omega$.

We now show that: for any $Y\in\Omega$, $Y\cap V(\phi)\neq \emptyset$ and $Y\cap \overline{V(\phi)}\neq \emptyset$.

Let $Y\in \Omega$, then by definition of $\Omega$, $m_\Delta(Y)\cap m_\Delta(\overline{V(\phi)}\cup Y)\in s$. By assumption and the fact that ultrafilters are closed under intersection, $m_\Delta(Y)\cap m_\Delta(\overline{V(\phi)}\cup Y)\cap V(\nabla\phi)\in s$. Since $s$ does not contain the empty set, we obtain $m_\Delta(Y)\cap m_\Delta(\overline{V(\phi)}\cup Y)\cap V(\nabla\phi)\neq \emptyset$. Then there is an element in $m_\Delta(Y)\cap m_\Delta(\overline{V(\phi)}\cup Y)\cap V(\nabla\phi)$, say $x$. From $x\in V(\nabla\phi)$, it follows that there exist $y,z\in R(x)$ with $y\in V(\phi)$ and $z\notin V(\phi)$, and then $z\in \overline{V(\phi)}\cup Y$. Then by $x\in m_\Delta(\overline{V(\phi)}\cup Y)$, we infer $y\in \overline{V(\phi)}\cup Y$, and thus $y\in Y$. Therefore $Y\cap V(\phi)\neq \emptyset$.\weg{ Moreover, since $m_\nabla(V(\phi))=V(\nabla\phi)$ and $Y\subseteq Y$, by Prop.~\ref{prop.usef}, we can infer that $x\in m_\Delta(V(\phi)\cup Y)$. From $y\in V(\phi)\cup Y$, it follows that $z\in V(\phi)\cup Y$. But $z\notin V(\phi)$, thus $z\in Y$,} Since $x\in m_\Delta(Y)$ and $xRy,xRz$, we infer $z\in Y$, and hence $z\in Y\cap \overline{V(\phi)}$, therefore $Y\cap \overline{V(\phi)}\neq \emptyset$.

This indicates that $\Omega\cup \{V(\phi)\}$ and $\Omega\cup \{\overline{V(\phi)}\}$ both have the finite intersection property. Using the Ultrafilter Theorem, there are $t,u\in Uf(S)$ such that $\Omega\cup\{V(\phi)\}\subseteq t$ and $\Omega\cup\{\overline{V(\phi)}\}\subseteq u$. Then by definition of $R^{ue}$, we conclude that $R^{ue}st$ and $R^{ue}su$ and $V(\phi)\in t$ and $\overline{V(\phi)}\in u$ (thus $V(\phi)\notin u$).

Therefore, for all $w\in S$, for all $\phi\in \mathcal{L}(\nabla)$, we have $\M,w\vDash\phi$ iff $w\in V(\phi)$ iff $V(\phi)\in \pi_w$ iff $ue(\M),\pi_w\vDash\phi$.
\end{proof}

\begin{proposition}
Let $\M$ be a model. Then $ue(\M)$ is $\mathcal{L}(\nabla)$-saturated.
\end{proposition}

\begin{proof}
Given any set $\Gamma\subseteq \mathcal{L}(\nabla)$, and any ultrafilter $s$ over $S$, we show that if $\Gamma$ is finitely satisfiable in the set of successors of $s$, then $\Gamma$ is satisfiable in the set of successors of $s$. W.l.o.g. we may assume that $R^{ue}(s)\neq \emptyset$, since otherwise the statement holds vacuously.

By definition of $R^{ue}$, $m_\nabla(X)\in s$ for some $X$. Suppose that $\Gamma$ is finitely satisfiable in the set of successors of $s$, to show that $\Gamma$ is satisfiable in the set of successors of $s$. For this, we define the following set
$$\Theta=\{V(\phi)\mid \phi\in\Gamma'\}\cup\{Y\mid m_\Delta(Y)\cap m_\Delta(\overline{X}\cup Y)\in s\},$$
where $\Gamma'$ is the set of finite conjunctions of formulas in $\Gamma$. In the following we will show that $\Theta$ has the finite intersection property.

The proof of Prop.~\ref{prop.ultrafilter1} has shown that $\{Y\mid m_\Delta(Y)\cap m_\Delta(\overline{X}\cup Y)\in s\}$ is closed under intersection. Moreover, $\{V(\phi)\mid \phi\in\Gamma'\}$ is closed under intersection, as shown below. Let $A,B\in \{V(\phi)\mid \phi\in\Gamma'\}$. Then $A=V(\phi)$ and $B=V(\psi)$ for some $\phi,\psi\in\Gamma'$. One may easily verify that $A\cap B=V(\phi\land\psi)$ and $\phi\land\psi\in\Gamma'$, and thus $A\cap B\in \{V(\phi)\mid \phi\in\Gamma'\}$.

Next we show that for any $\phi\in\Gamma'$ and $Y\subseteq S$ for which $m_\Delta(Y)\cap m_\Delta(\overline{X}\cup Y)\in s$, we have $V(\phi)\cap Y\neq \emptyset$. By $\phi\in\Gamma'$ and supposition, there is an ultrafilter $t$ over $S$ such that $R^{ue}st$ and $ue(\M),t\vDash\phi$. By Prop.~\ref{prop.ultrafilter1}, $V(\phi)\in t$; by $R^{ue}st$ and the definition of $R^{ue}$, $Y\in t$. Since $t$ is closed under intersection, we obtain $V(\phi)\cap Y\in t$. Since $t$ does not contain the empty set, $V(\phi)\cap Y\neq \emptyset$.

We have thus shown that $\Theta$ has the finite intersection property. Then by the Ultrafilter Theorem, $\Theta$ can be extended to an ultrafilter $u$. From $m_\nabla(X)\in s$ and $\{Y\mid m_\Delta(Y)\cap m_\Delta(\overline{X}\cup Y)\in s\}\subseteq u$, it follows that $R^{ue}su$; moreover, from $\{V(\phi)\mid \phi\in\Gamma'\}\subseteq u$, it follows that $ue(\M),u\vDash\Gamma$: since for any $\phi\in \Gamma$, $\phi\in\Gamma'$, then $V(\phi)\in u$, applying Prop.~\ref{prop.ultrafilter1} we derive $ue(\M),u\vDash\phi$.
\end{proof}

Since the class of $\mathcal{L}(\nabla)$-saturated models has the Hennessy-Milner property~\cite[Prop.~3.9]{Fanetal:2014}, we obtain the main result of this section:  $\mathcal{L}(\nabla)$-equivalence can be thought of as $\mathcal{L}(\nabla)$-bisimilarity between ultrafilter extensions.

\begin{theorem}
Let $\M$ and $\M'$ be models and $w\in \M$, $w'\in\M'$. Then
$$(\M,w)\equiv_{\mathcal{L}(\nabla)}(\M',w')\Longleftrightarrow(ue(\M),\pi_w)\bis_\Delta(ue(\M'),\pi_{w'}).$$
\end{theorem}

\weg{\begin{proposition}
For all $\phi\in\mathcal{L}(\nabla)$, if $ue(\mathcal{F})\vDash\phi$, then $\mathcal{F}\vDash\phi$.
\end{proposition}

\begin{proof}
Suppose that $\mathcal{F}\nvDash\phi$, then there is a valuation $V$ and a state $w$ such that $\lr{\mathcal{F},V},w\nvDash\phi$. By Prop.~\ref{prop.ultrafilter}, we infer that $\lr{ue(\mathcal{F}),V^{ue}},\pi_w\nvDash\phi$. Therefore, $ue(\mathcal{F})\nvDash\phi$.
\end{proof}}

\subsubsection{$\forall$-version}

Another version is inspired by the canonical model given in~\cite[p.~232]{DBLP:journals/ndjfl/Kuhn95}, which is in fact equal to the canonical model originally given in~\cite{Humberstone95}, as shown in~\cite{Fan:2017b}.

The canonical model $\M=\lr{S,R,V}$ for minimal contingency logic in~\cite{DBLP:journals/ndjfl/Kuhn95} is defined as follows:
\begin{itemize}
\item $S=\{\Sigma\mid \Sigma\text{ is a maximal consistent set}\}$.
\item $R\Sigma\Gamma$ iff for all $\phi$, if for all $\psi$, $\Delta(\psi\vee\phi)\in \Sigma$, then $\phi\in \Gamma$.
\item For each $p\in\BP$ and $\Sigma\in S$, $\Sigma\in V(p)$ iff $p\in \Sigma$.
\end{itemize}

The following result is shown in~\cite[Lemma~2]{DBLP:journals/ndjfl/Kuhn95}.
\begin{proposition}
Let $\Sigma\in S$. Then for all $\phi\in\mathcal{L}(\nabla)$, for each $p\in\BP\cup\{\nabla\phi\}$, we have
$$p\in\Sigma\iff \Sigma\in V(p).$$
\end{proposition}

Prop.~\ref{prop.general} then gives us the following result.
\begin{proposition}
Let $\Sigma\in S$. Then for all $\phi\in\mathcal{L}(\nabla)$, we have
$$\phi\in\Sigma\iff V(\phi)\in \pi_\Sigma.$$
\end{proposition}



The following is direct from Def.~\ref{def.m-Delta}.
\begin{proposition}\label{prop.intersection}
For all $X,Y$, $m_\Delta(X)\cap m_\Delta(Y)\subseteq m_\Delta(X\cap Y)$.
\end{proposition}

Once again, with the method in Sec.~\ref{sec.normalmodal} in hand, we can obtain the ultrafilter extension in Kripke contingency logic. We call it `$\forall$-version', since the accessibility relation $R^{ue}$ has the $\forall$-form.

\begin{definition}[Ultrafilter Extension] Let $\M=\lr{S,R,V}$ be a model. The triple $ue(\M)=\lr{Uf(S),R^{ue},V^{ue}}$ is the {\em ultrafilter extension} of $\M$, if
\begin{itemize}
\item $Uf(S)=\{s\mid s\text{ is an ultrafilter over }S\}$.
\item $R^{ue}st$ iff for all $Y$, if $m_\Delta(X\cup Y)\in s$ for all $X$, then $Y\in t$.
\item $V^{ue}(p)=\{s\in Uf(S)\mid V(p)\in s\}$.
\end{itemize}
\end{definition}

\begin{proposition}\label{prop.ultrafilter2}
For all formulas $\phi\in\mathcal{L}(\nabla)$ and all ultrafilters $s$ over $S$, we have $V(\phi)\in s$ iff $ue(\M),s\vDash\phi$. As a corollary, for all $w\in S$, we have $(\M,w)\equiv_{\mathcal{L}(\nabla)}(ue(\M),\pi_w)$.
\end{proposition}

\begin{proof}
By induction on $\phi\in\mathcal{L}(\nabla)$. We only consider the case $\nabla\phi$.

First assume that $ue(\M),s\vDash\nabla\phi$. Then there exist $t,u\in Uf(S)$ such that $R^{ue}st$ and $R^{ue}su$ and $ue(\M),t\vDash\phi$ and $ue(\M),u\nvDash\phi$. By induction hypothesis, $V(\phi)\in t$ (i.e. $\overline{V(\phi)}\notin t$) and $V(\phi)\notin u$. By definition of $R^{ue}$, $m_\Delta(X\cup \overline{V(\phi)})\notin s$ for some $X$ and $m_\Delta(Y\cup V(\phi))\notin s$ for some $Y$. Thus $\in s$ and $m_\nabla(\overline{Y}\cap \overline{V(\phi)})\in s$. Since $s$ is closed under intersection, we have $m_\nabla(\overline{X}\cap V(\phi))\cap m_\nabla(\overline{Y}\cap \overline{V(\phi)})\in s$. Since $s$ is closed under supersets, by Prop.~\ref{prop.subseteq}, we infer that $m_\nabla(V(\phi))\in s$, and thus $V(\nabla\phi)\in s$.

Conversely, assume that $V(\nabla\phi)\in s$, to show $ue(\M),s\vDash\nabla\phi$. By induction hypothesis, we need to find two ultrafilters $t,u$ over $S$ such that $R^{ue}st$ and $R^{ue}su$ and $V(\phi)\in t$ and $V(\phi)\notin u$.

Now define $\Xi=\{Y\mid m_\Delta(X\cup Y)\in s\text{ for all }X\}$. We first show that $\Xi$ is closed under intersection. Let $Y,Z\in \Xi$. Then $m_\Delta(X\cup Y)\in s$ and $m_\Delta(X\cup Z)\in s$ for all $X$. Since $s$ is closed under intersection, $m_\Delta(X\cup Y)\cap m_\Delta(X\cup Z)\in s$. Since $m_\Delta(X\cup Y)\cap m_\Delta(X\cup Z)\subseteq m_\Delta(X\cup (Y\cap Z))$ (by Prop.~\ref{prop.intersection}) and $s$ is closed under supersets, $m_\Delta(X\cup (Y\cap Z))\in s$. Since $X$ is arbitrary, we have $Y\cap Z\in \Xi$.

We now show that

For any $Y\in\Xi$, $Y\cap V(\phi)\neq \emptyset$ and $Y\cap \overline{V(\phi)}\neq \emptyset$.

Let $Y\in \Xi$. Then $m_\Delta(X\cup Y)\in s\text{ for all }X$. By assumption and the fact that $s$ is closed under intersection, $m_\Delta(X\cup Y)\cap V(\nabla\phi)\in s$ for all $X$. Since $s$ does not contain the empty set, we have $m_\Delta(X\cup Y)\cap V(\nabla\phi)\neq\emptyset$ for all $X$. Thus there is an $x$ such that $x\in m_\Delta(X\cup Y)\cap V(\nabla\phi)$ for all $X$. By $x\in V(\nabla\phi)$, there exist $y,z\in R(x)$ such that $y\in V(\phi)$ and $z\notin V(\phi)$ (thus $z\in \overline{V(\phi)}$). We have also $x\in m_\Delta(X\cup Y)$ for all $X$. Letting $X=\overline{V(\phi)}$, we get $x\in m_\Delta(\overline{V(\phi)}\cup Y)$, and then $y\in \overline{V(\phi)}\cup Y$ iff $z\in\overline{V(\phi)}\cup Y$, then $y\in \overline{V(\phi)}\cup Y$, and hence $y\in Y$. This implies that $Y\cap V(\phi)\neq \emptyset$. Now letting $X=\emptyset$, we get $x\in m_\Delta(Y)$, and then $y\in Y$ iff $z\in Y$, therefore $z\in Y$, which implies that $Y\cap \overline{V(\phi)}\neq \emptyset$.

Then similar to the corresponding part of the proof of Prop.~\ref{prop.ultrafilter1}, we can find two states $t,u\in Uf(S)$ with $R^{ue}st$ and $R^{ue}su$ and $V(\phi)\in t$ and $V(\phi)\notin u$.
\end{proof}

\begin{proposition}
Let $\M$ be a model. Then $ue(\M)$ is $\mathcal{L}(\nabla)$-saturated.
\end{proposition}

\begin{proof}
Given any set $\Gamma\subseteq \mathcal{L}(\nabla)$, and any ultrafilter $s$ over $S$, suppose that $\Gamma$ is finitely satisfiable in the set of successors of $s$, to show that $\Gamma$ is satisfiable in the set of successors of $s$. For this, we define the following set
$$\Theta=\{V(\phi)\mid \phi\in\Gamma'\}\cup\{Y\mid m_\Delta(X\cup Y)\in s\text{ for all }X\},$$
where $\Gamma'$ is the set of finite conjunctions of formulas in $\Gamma$. In the following we will show that $\Theta$ has the finite intersection property.

Prop.~\ref{prop.ultrafilter2} has shown that $\{Y\mid m_\Delta(X\cup Y)\in s\text{ for all }X\}$ is closed under intersection. Moreover, $\{V(\phi)\mid \phi\in\Gamma'\}$ is closed under intersection, as shown below. Let $A,B\in \{V(\phi)\mid \phi\in\Gamma'\}$. Then $A=V(\phi)$ and $B=V(\psi)$ for some $\phi,\psi\in\Gamma'$. One may easily verify that $A\cap B=V(\phi\land\psi)$ and $\phi\land\psi\in\Gamma'$, and thus $A\cap B\in \{V(\phi)\mid \phi\in\Gamma'\}$.

Next we show that for any $\phi\in\Gamma'$ and $Y\subseteq S$ for which $m_\Delta(X\cup Y)\in s\text{ for all }X$, we have $V(\phi)\cap Y\neq \emptyset$. By $\phi\in\Gamma'$ and supposition, there is an ultrafilter $t$ over $S$ such that $R^{ue}st$ and $ue(\M),t\vDash\phi$. By Prop.~\ref{prop.ultrafilter2}, $V(\phi)\in t$; by $R^{ue}st$ and the definition of $R^{ue}$, $Y\in t$. Since $t$ is closed under intersection, we obtain $V(\phi)\cap Y\in t$. Since $t$ does not contain the empty set, $V(\phi)\cap Y\neq \emptyset$.

We have thus shown that $\Theta$ has the finite intersection property. Now by the Ultrafilter Theorem, $\Theta$ can be extended to an ultrafilter $u$. From $\{Y\mid m_\Delta(X\cup Y)\in s\text{ for all }X\}\subseteq u$, it follows that $R^{ue}su$; moreover, from $\{V(\phi)\mid \phi\in\Gamma'\}\subseteq u$, it follows that $ue(\M),u\vDash\Gamma$: since for any $\phi\in \Gamma$, $\phi\in\Gamma'$, then $V(\phi)\in u$, applying Prop.~\ref{prop.ultrafilter2} we derive $ue(\M),u\vDash\phi$.
\end{proof}

Since the class of $\mathcal{L}(\nabla)$-saturated models has the Hennessy-Milner property~\cite[Prop.~3.9]{Fanetal:2014}, we obtain the main result of this article: modal equivalence can be thought of as $\mathcal{L}(\nabla)$ equivalence between ultrafilter extensions.

\begin{theorem}
Let $\M$ and $\M'$ be models and $w\in \M$, $w'\in\M'$. Then
$$(\M,w)\equiv_{\mathcal{L}(\nabla)}(\M',w')\Longleftrightarrow(ue(\M),\pi_w)\bis_\Delta(ue(\M'),\pi_{w'}).$$
\end{theorem}

\subsection{Neighborhood contingency logics}\label{sec.neighborhoodcontingency}

Given a neighborhood model $\M=\lr{S,N,V}$, the non-contingency operator is interpreted in the following.
\[
\begin{array}{lll}
\M,s\Vdash\Delta\phi&\iff&V(\phi)\in N(s)\text{ or }\overline{V(\phi)}\in N(s)\\
\end{array}
\]
Where $V(\phi)=\{w\in S\mid \M,w\Vdash\phi\}$.

Define $m_C(X)=\{s\in S\mid X\in N(s)\text{ or }\overline{X}\in N(s)\}$. It is obvious that $m_C(X)=m_C(\overline{X})$ and $m_C(V(\phi))=V(\Delta\phi)$.

Recall that in the completeness of classical contingency logic $\mathbb{CCL}$~\cite{FanvD:neighborhood}, the canonical model $\M=\lr{S,N,V}$ is defined as follows:
\begin{itemize}
\item $S=\{\Sigma\mid \Sigma\text{ is a maximal consistent set}\}$
\item $N(\Sigma)=\{|\phi|\mid \Delta\phi\in \Sigma\}$, where $|\phi|=\{\Sigma\in S\mid \phi\in\Sigma\}$
\item For each $p\in\BP$, $\Sigma\in V(p)$ iff $p\in\Sigma$.
\end{itemize}

The following is a standard result in classical contingency logic~\cite[Lemma~1]{FanvD:neighborhood}.
\begin{proposition}\label{prop.comp-classical-modal} Let $\Sigma\in S$. Then for all $\phi\in\mathcal{L}(\nabla)$, for each $p\in\BP\cup\{\nabla\phi\}$, we have
$$p\in\Sigma\text{ iff }\Sigma\in V(p).$$
\end{proposition}

Then Prop.~\ref{prop.general} gives us the following result.
\begin{proposition}\label{prop.modal}
Let $\Sigma\in S$. Then for all $\phi\in\mathcal{L(\nabla)}$, we have
$$\phi\in\Sigma\text{ iff }V(\phi)\in \pi_\Sigma.$$
\end{proposition}

Now applying the method in Sec.~\ref{sec.normalmodal}, we obtain the notion of ultrafilter extensions in classical contingency logics.

\begin{definition}[Ultrafilter extensions] Let $\M=\lr{S,N,V}$ be a neighborhood model. The triple $ue(\M)=\lr{Uf(S),N^{ue},V^{ue}}$ is the {\em ultrafilter extension} of $\M$, if
\begin{itemize}
\item $Uf(S)=\{s\mid s\text{ is an ultrafilter over }S\}$
\item $N^{ue}(s)=\{\widehat{X}\mid m_C(X)\in s\}$, where $\widehat{X}=\{s\in Uf(S)\mid X\in s\}$.
\item For each $p\in \BP$ and $s\in Uf(S)$, $s\in V^{ue}(p)$ iff $V(p)\in s$.
\end{itemize}
\end{definition}


For all $s\in Uf(S)$ and for all $X\subseteq S$, we can show $\overline{\hat{X}}=\hat{\overline{X}}$. With this and the fact that $m_C(X)=m_C(\overline{X})$, we can show that $N^{ue}$ is closed under complements, in that for all $U\subseteq Uf(S)$, if $U\in N^{ue}(s)$, then $\overline{U}\in N^{ue}(s)$.

\begin{proposition}\label{prop.ultrafilter-delta2}
For all formulas $\phi\in\mathcal{L}(\nabla)$ and all ultrafilters $s$ over $S$, we have $V(\phi)\in s$ iff $ue(\M),s\Vdash\phi$. As a corollary, for all $w\in S$, we have $(\M,w)\equiv_{\mathcal{L}(\nabla)}(ue(\M),\pi_w)$.
\end{proposition}

\begin{proof}
By induction on $\phi\in\mathcal{L}(\nabla)$. The only nontrivial case is $\Delta\phi$. For this, we have the following equivalences:
\[
\begin{array}{ll}
&ue(\M),s\vDash\Delta\phi\\
\iff & V^{ue}(\phi)\in N^{ue}(s)\text{ or }\overline{V^{ue}(\phi)}\in N^{ue}(s)\\
\stackrel{\text{IH}}\iff&\hat{V(\phi)}\in N^{ue}(s)\text{ or }\overline{\hat{V(\phi)}}\in N^{ue}(s)\\
\iff&\hat{V(\phi)}\in N^{ue}(s)\text{ or }\hat{\overline{V(\phi)}}\in N^{ue}(s)\\
\iff & m_C(V(\phi))\in s\text{ or }m_C(\overline{V(\phi)})\in s\\
\iff & m_C(V(\phi))\in s\\
\iff &V(\Delta\phi)\in s\\
\end{array}
\]
\weg{\[
\begin{array}{ll}
&ue(\M),s\vDash\Delta\phi\\
\iff & V^{ue}(\phi)\in N^{ue}(s)\text{ or }\overline{V^{ue}(\phi)}\in N^{ue}(s)\\
\stackrel{\text{IH}}\iff&\{s\in Uf(S)\mid V(\phi)\in s\}\in N^{ue}(s)\text{ or }\overline{\{s\in Uf(S)\mid V(\phi)\in s\}}\in N^{ue}(s)\\
\stackrel{(\star)}\iff&\{s\in Uf(S)\mid V(\phi)\in s\}\in N^{ue}(s)\text{ or }\{s\in Uf(S)\mid \overline{V(\phi)}\in s\}\in N^{ue}(s)\\
\iff & m_C(V(\phi))\in s\text{ or }m_C(\overline{V(\phi)})\in s\\
\stackrel{(\star\star)}\iff & m_C(V(\phi))\in s\\
\iff &V(\Delta\phi)\in s\\
\end{array}
\]}
\weg{\[
\begin{array}{ll}
&V(\Delta\phi)\in s\\
\iff & m_N(V(\phi))\in s\\
\iff &\{s\in Uf(S)\mid V(\phi)\in s\}\in N^{ue}(s)\\
\stackrel{\text{IH}}\iff&V^{ue}(\phi)\in N^{ue}(s)\\
\iff & ue(\M),s\vDash\Delta\phi\\
\end{array}
\]
In the proof of $(\star)$, $\overline{\{s\in Uf(S)\mid V(\phi)\in s\}}=\{s\in Uf(S)\mid \overline{V(\phi)}\in s\}$ is because $s$ is an ultrafilter. $(\star\star)$ holds since $m_C(X)=m_C(\overline{X})$ for all $X$.}
\end{proof}

In the remainder of this section, we show that $\mathcal{L}(\nabla)$-equivalence can be characterized as nbh-$\Delta$-bisimilarity-somewhere-else.\footnote{As for the definition of nbh-$\Delta$-bisimilarity, refer to~\cite[Def.~9]{Bakhtiarietal:2017}; this definition is simplified in~\cite{Fan:2017a}.} For this, we need to introduce a notion of $\Delta$-saturation for $\mathcal{L}(\nabla)$ in the neighborhood setting, called `$\mathcal{L}_\Delta$-saturation' in~\cite[Def.~11]{Bakhtiarietal:2017}.

\begin{definition}[$\Delta$-saturation]
Let $\M=\lr{S,N,V}$ be a neighborhood model. A set $X\subseteq S$ is $\Delta$-compact, if every set of $\mathcal{L}(\nabla)$-formulas that is finitely satisfiable in $X$ is itself also satisfiable in $X$. $\M$ is said to be {\em $\Delta$-saturated}, if for all $s\in S$ and all $\equiv_{\mathcal{L}(\nabla)}$-closed neighborhoods $X\in N(s)$, both $X$ and $\overline{X}$ are $\Delta$-compact.
\end{definition}

It is shown in~\cite[Thm.1]{Bakhtiarietal:2017} that the class of $\Delta$-saturated models is a Hennessy-Milner class. That is, on $\Delta$-saturated models $\M$ and $\M'$ and states $s$ in $\M$ and $s'$ in $\M'$, if $(\M,s)\equiv_{\mathcal{L}(\nabla)}(\M',s')$, then $(\M,s)\sim_\Delta(\M',s').$


\begin{proposition}
Let $\M$ be a neighborhood model. Then $ue(\M)$ is $\Delta$-saturated.
\end{proposition}

\begin{proof}
Given any $s\in Uf(S)$ and any $\equiv_{\mathcal{L}(\nabla)}$-closed neighborhood $\hat{X}\in N^{ue}(s)$, to show $\hat{X}$ and $\overline{\hat{X}}$ are both $\Delta$-compact. Since $N^{ue}$ is closed under complements and $\overline{\hat{X}}=\hat{\overline{X}}$, it is sufficient to show that $\hat{X}$ is $\Delta$-compact~\cite[p.~6]{Fan:2017a}.

Suppose that $\Pi\subseteq \mathcal{L}(\nabla)$ is finitely satisfiable in $\hat{X}$. This means that for any finite set $\{\phi_1,\cdots,\phi_m\}\subseteq \Pi$, there exists $u\in \hat{X}$ such that $\M^{ue},u\vDash\phi_1\land\cdots\land\phi_m$. By Prop.~\ref{prop.ultrafilter-delta2}, $V(\phi_1)\cap\cdots\cap V(\phi_m)=V(\phi_1\land\cdots\land\phi_m)\in u$. By $u\in\hat{X}$, $X\in u$. Since $u$ is closed under intersection, $V(\phi_1)\cap\cdots\cap V(\phi_m)\cap X\in u$. Since $u$ does not contain the empty set, it follows that $V(\phi_1)\cap\cdots\cap V(\phi_m)\cap X\neq\emptyset$.

Because $\phi_1,\cdots,\phi_m$ are arbitrary, the set $\{V(\phi)\mid \phi\in\Pi\}\cup\{X\}$ has the finite intersection property. By the Ultrafilter Theorem, there is an $u'\in Uf(S)$ such that $\{V(\phi)\mid \phi\in\Pi\}\cup\{X\}\subseteq u'$. From $\{V(\phi)\mid \phi\in\Pi\}\subseteq u'$ and Prop.~\ref{prop.ultrafilter-delta2}, it follows that $ue(\M),u'\vDash\Pi$; from $X\in u'$, it follows that $u'\in \hat{X}$. Therefore, $\Pi$ is satisfiable in $\hat{X}$.
\end{proof}

Since the class of $\Delta$-saturated models is a Hennessy-Milner class, we obtain a characterization of $\mathcal{L}(\nabla)$-equivalence as nbh-$\Delta$-bisimilarity-somewhere-else --- namely, between ultrafilter extensions.

\begin{theorem}
Let $\M$ and $\M'$ be neighborhood models. Then
$$(\M,w)\equiv_{\mathcal{L}(\nabla)}(\M',w')\iff (ue(\M),\pi_w)\sim_\Delta(ue(\M'),\pi_{w'}).$$
\end{theorem}

\section{Conclusion}\label{sec.conclusion}



In this paper, we proposed a uniform method of constructing ultrafilter extensions out of canonical models, based on the similarity between ultrafilters and maximal consistent sets. We illustrated this method with ultrafilter extensions of models for normal modal logics and for classical modal logics, which can help us understand why the known ultrafilter extensions are so defined. We then applied it to obtain ultrafilter extensions of any Kripke model and of any neighborhood model for contingency logic. Our results also hold for multi-modal cases.


Although we only investigated ultrafilter extensions of any Kripke/neighborhood model for standard modal logic and contingency logic, we believe our method also works for many other logics, special models (monotonic models, regular models, reflexive models, etc.), and many other semantics (algebraic semantics, coalgebraic semantics, etc.). Once we have the canonical model of a logic, we can construct the notion of ultrafilter extension, by using the above-mentioned uniform method, in an automatic way. With suitable notions of bisimilarity and saturation, we can show the ultrafilter extension is as desired.



\bibliographystyle{plain}
\bibliography{biblio2018,neighbor,biblio2016}

\end{document}